\documentclass[12pt]{amsart}
\usepackage{graphicx}
\usepackage{amssymb,latexsym,amsmath}
\usepackage{amsfonts,amsbsy,bm}
\DeclareMathOperator{\sign}{sign}

\begin{document}
\title[Embeddings and Duality]{Dyadic Martingale Hardy-amalgam spaces: Embeddings and Duality}
\author{Justice Sam Bansah and Beno\^it F. Sehba}
\address{Department of Mathematics, University of Ghana, P. O. Box LG 62 Legon, Accra, Ghana}
\email{fjaccobian@gmail.com}
\address{Department of Mathematics, University of Ghana, P. O. Box LG 62 Legon, Accra, Ghana}
\email{bfsehba@ug.edu.gh}

\newtheorem{theorem}[subsection]{Theorem}
\newtheorem{proposition}[subsection]{Proposition}
\newtheorem{lemma}[subsection]{Lemma}
\newtheorem{corollary}[subsection]{Corollary}
\theoremstyle{definition}
\newtheorem{definition}[subsection]{Definition}
\newtheorem{remark}[subsection]{Remark}
\newtheorem{example}[subsection]{Example}
\newtheorem{solution}[subsection]{Solution}

\begin{abstract}
We present in this paper some embeddings of various dyadic martingale Hardy-amalgam spaces $H^S_{p,q},\,\, H^s_{p,q},\,\,H^*_{p,q},\,\,\mathcal{Q}_{p,q}$ and $\mathcal{P}_{p,q}$ of the real line. In the same settings, we characterize the dual of $H^s_{p,q}$ for large $p$ and $q$. We also introduce a Garsia-type space $\mathcal{G}_{p,q}$ and characterize its dual space. 
\end{abstract}
\maketitle
\section{Introduction}\label{sec1}
Our setting is a quadruplet $(\mathbb{R},\mathcal{F},\{\mathcal{F}\}_{n\ge 0},\mathbb{P})$ where $\mathcal{F}_n$ is the sigma algebra generated by all dyadic intervals of $\mathbb{R}$ of lentgh $2^{-n}$, $\mathbb{P}$ a probability measure, and $\mathcal{F}$ stands for the sigma algebra generated by the union of the $\mathcal{F}_n$s. In this work, we are interested in the study of the embeddings relations between the recently introduced martingale Hardy-amalgam spaces in \cite{BanSeh} and the characterization of their dual spaces for large exponents. We also define a natural extension of the Garsia space to the Wiener amalgam setting for which we characterize the dual space. 
\vskip .2cm
Martingale inequalities have proven to be very useful in various applications. For instance, the justification of martingale convergence theorems for both forward and backward convergences have been established by applying classical martingale inequalities (see \cite{Hyt}). In Fourier analysis, we have the involvement of martingale inequalities in the establishment of the boundedness of the maximal Fej\'er operator (see \cite{Xie}). Martingale inequalities also play some important roles in the study of properties of Brownian motions (see for example \cite{BurkGund}).
\vskip .2cm
In the past few decades, there have been various studies about the embeddings of classical martingale Hardy spaces. Some of these discussions can be found in \cite{Garsia, Long, Ferenc}. There are many other important inequalities involving martingales that have been proved and applied in the literature (see for example \cite{Burk, BurkGund, Izu, IzuKaza, Kaza,Ferenc}). For instance in \cite{Ferenc} we can find a discussion on the Doob's inequality, the Convexity and Concavity inequalities for martingales. We can also find a discussion on norm inequalities for operators of matrix type on martingales and a proof of Burkholder-Davis-Gundy inequality in \cite{BurkGund,Ferenc}. A discussion on the weighted norm inequality similar to the Burkholder-Davis-Gundy inequality can also be found in \cite{Izu} and inequalities of operators of non-matrix type on martingales with a weighted probability measure are discussed in \cite{Kaza}. It is also interesting to mention that an analogue of weighted norm inequality for the Hardy maximal function result is also valid in the setting of martingale theory (see \cite{IzuKaza}). Some of these classical results will be very useful in this paper and for consistency purposes, we will restate these classical results when needed in the appropriate section below. 

Let $\lambda = (\lambda_n)_{n\in\mathbb{N}}$ be an adapted sequence (this will be made clear in the next section) and let $d_n\lambda=\lambda_n-\lambda_{n-1}.$ Then $\lambda$ is said to be $L_p$-variation integrable bounded if $\|\sum_n|d_n\lambda|\|_{L_p}<\infty.$ It is also said to be $L_p$-jump bounded if $\sup_n|d_nf|\in L_p$. The space of all martingales $\lambda$ that are $L_p$-variation integrable bounded is often referred to as variation integrable space (see \cite{Long}). It is also referred to as the Garcia space and it is denoted $\mathcal{G}_p$ (see \cite{Ferenc}). Also the space all martingales $\lambda$ that are $L_p$-jump bounded is simply referred to as the Jump bounded space (see \cite{Long}). This is the space denoted as $\mathcal{BD}_p$ in \cite{Ferenc}. The Garcia space is shown to be a component of the Davis decompositions of martingales in the classical martingale Hardy spaces (see \cite{Long}). It is also established that the dual space of the Garcia space is the Jump bounded space (see \cite{Long,Ferenc}).
\vskip .2cm
We recall (see \cite{BanSeh}) that if $T$ is either the quadratic variation ($S$), the conditional quadratic variation ($s$) or the maximal function ($Mf:f^*$), then the corresponding martingale Hardy-amalgam space $H_{p,q}^T$ is the space of all martingales $f$ such that $T(f)$ belongs to the Wiener amalgam space $L_{p,q}$, $0<p,q\le \infty$. The amalgam space of predictive martingales $\mathcal{P}_{p,q}$ and the amalgam space of martingales with predictive quadratic variation $\mathcal{Q}_{p,q}$ are defined as in the classical case (see \cite{Ferenc}) by just replacing the Lebesgue space in the definition by the amalgam space $L_{p,q}$.
\vskip .2cm

In this paper, we extend the embeddings of the classical martingale Hardy spaces, (see \cite[Theorem 2.11]{Ferenc}), to the martingale Hardy-amalgam spaces.  We will also extend the Doob's inequality and the Burkholder-Davis-Gundy inequality of the classical martingale Hardy spaces to the martingale Hardy-amalgam spaces.We shall also introduce the space of $L_{p,q}$-variation integrable bounded martingales $\mathcal{G}_{p,q}$ and the space $L_{p,q}$-jump bounded martingales $\mathcal{BD}_{p,q},$ and we shall also refer to these spaces as the Garcia (or the variation integrable) space and Jump bounded space respectively. We will then establish that the dual of $\mathcal{G}_{p,q}$ is $\mathcal{BD}_{p',q'}$ where $(p,p')$ and $(q,q')$ are conjugate pairs.  In the classical case, the space $\mathcal{G}_p$ played an important role in establishing the famous Fefferman's inequality (see for example \cite{Long}). This motivated us to introduce $\mathcal{G}_{p,q}$ and study one property of this space, which is its duality.  
\vskip .2cm
It was established in \cite{BanSeh} that the dual of $H^s_{p,q}$ when $0<p\le q\le1$ is a Campanato-type space while the case $1<p,q<\infty$ was left open. We prove in this paper that for $1<q\le p\le 2$ or $2\le p\le q<\infty$, the duality of $H^s_{p,q}$ identifies with $H^s_{p',q'}$.
\vskip .2cm
We note that the setting of \cite{BanSeh} is more general. In a previous version of \cite{BanSeh}, we have tried to solve embeddings and duality problem but this quickly appeared to be a difficult task. A key argument in the dyadic setting is Lemma \ref{lem:idmartingale} which allows us to obtain the results of this paper. For the embeddings, we combine this lemma with known classical results. The proofs of duality results are more demanding but again Lemma \ref{lem:idmartingale} is relevant here as it is used in the proof of the Doob's inequality which is in its turn used in our proofs.
\vskip .2cm
The outline of this paper is as follows. In Section \ref{sec2}, we will get familiar with notations and recall the various definitions appropriate for this work. The main results in this work are presented in Section \ref{sec3}. In Section \ref{sec4}, we provide proof for the first result of this work which is the extension of the classical martingale Hardy spaces embeddings to the martingale Hardy-amalgam spaces. We also provide proofs for the extension of the Doob's inequality and Burkholder-Davis-Gundy inequalities establishing the second and third result of this work. In Section \ref{sec5}, we will characterize the dual space of $H^s_{p,q}$ for $1<p\le q<\infty$ and also identify the dual space of $\mathcal{G}_{p,q}$. We will close the paper with a conclusion where we will make some other observations.

\section{Notations and Basic Definitions}\label{sec2}
In this section, we introduce the necessary definitions and recall the various martingale Hardy-amalgam spaces we shall consider in this paper. We will also state some important classical results, such as Doob's inequality and the Burkholder-Davis-Gundy inequality that we will need later in this work. 

Let $\mathbb{R}$ be the set of real numbers and consider the following dyadic intervals of $\mathbb{R};$ $$I_{n,k} = \left[\frac{k}{2^n},\frac{k+1}{2^n}\right)\,\, n\in\mathbb{N},\,\,k\in\mathbb{Z}.$$ Let $\mathcal{D}_n=\{I_{n,k},\,\,k\in\mathbb{Z}\},\,\,n\in\mathbb{N}.$ 
Let $\mathcal{F}_n=\sigma(\mathcal{D}_n)$ be the $\sigma-$algebra generated by $\mathcal{D}_n.$ Then $\{\mathcal{F}_n\}_{n\in\mathbb{N}}$ is a (dyadic) filtration. 
With this, we define the probability space as $$\mathbf{Ps}:=(\mathbb{R}, \mathcal{F}, \{\mathcal{F}_n\}_{n\in\mathbb{N}}, \mathbb{P})$$ where $\mathbb{P}$ is the probability measure and $\mathcal{F}_n\subseteq\mathcal{F}$ for all $n\in\mathbb{N}$. Thus all the martingales defined in this paper are with respect to this probability space with the underlying filtration $\{\mathcal{F}_n\}_{n\in\mathbb{N}}.$ Let $J_{k,n,j}\in \mathcal{D}_n$ be the dyadic interval defined as $$J_{k,n,j} = \left[\frac{k+j2^n}{2^n},\frac{k+1+j2^n}{2^n}\right).$$ Then \begin{eqnarray}\label{dia1}A_j = [j,j+1) = \bigcup_{k=0}^{2^n-1}J_{k,n,j}\end{eqnarray} 
Therefore, $A_j\in\mathcal{F}_n$ for all $n.$ 

Note that the $A_j$'s are dyadic intervals. Also, $A_i\cap A_j=\emptyset$ for $i\ne j$ and $\bigcup_jA_j=\mathbb{R}.$ 
\vskip .2cm
Let $L_p$ denote the classical Lebesgue space and let $\ell_q$ denote the sequence space. For $f\in L_p$, we will be using the notation
$$\|f\|_p:=\|f\|_{L_p}:=\left(\int_{\mathbb{R}}|f|^pd\mathbb{P}\right)^{1/p}.$$
The amalgam space of $L_p$ and $\ell_q$ is defined as the space $$L_{p,q}(\mathbb{R}) = \left\{f : \sum_j\|f\mathbf{1}_{A_j}\|_p^q<\infty\right\}$$ equipped with the (quasi)-norm $$\|f\|_{L_{p,q}(\mathbb{R})}=\left(\sum_{j\in\mathbb{Z}}\left(\int_{\mathbb{R}}|f|^p\mathbf{1}_{A_j}\mathrm{d}\mathbb{P}\right)^{\frac{q}{p}}\right)^{\frac{1}{q}}$$ for $0<p,q<\infty$ and $$\|f\|_{L_{p,\infty}(\mathbb{R})}=\sup_{j\in\mathbb{Z}}\left(\int_{\mathbb{R}}|f|^p\mathbf{1}_{A_j}\mathrm{d}\mathbb{P}\right)^{\frac{1}{p}}$$ for $0<p<\infty$. It is interesting to note that $L_{p,p}(\mathbb{R})=L_p(\mathbb{R}).$ We refer the interested reader to (\cite{Justin, Fournier, Heil, Holland}) for more on amalgam spaces. 

Let $\mathcal{M}$ be the spaces of all martingales defined on $\mathbf{Ps}$ relative to the underlying filtration $\mathcal{F}_n.$ For $f=(f_n)_{n\in\mathbb{N}}\in\mathcal{M},$ we define the martingale difference as $d_nf = f_n - f_{n-1}$ and we will agree that $f_0=0$ and $d_0f=0.$ Let $\|\cdot\|_p$ denote the usual $L_p-$norm for $0<p\le\infty.$ Then $f=(f_n)_{n\in\mathbb{N}}\in\mathcal{M}$ is said to be $L_p$ bounded if $$\|f\|_p:=\sup_{n\in\mathbb{N}}\|f_n\|_p<\infty$$ and we define the maximal function, $f^*,$ or $M(f)$ of $f$ as $$M(f) = f^*:=\sup_{n\in\mathbb{N}}|f_n|.$$ Let $\mathbb{E}$ and $\mathbb{E}_n$ be the expectation and the conditional expectation operators respectively. Then the following measurable functions are well defined (see for example \cite{BurkGund});  $$S(f) = \left(\sum_{n\in\mathbb{N}}|d_nf|^2 \right)^{\frac{1}{2}}\quad\mbox{and}\quad s(f) = \left(\sum_{n\in\mathbb{N}}\mathbb{E}_{n-1}|d_nf|^2 \right)^{\frac{1}{2}}$$ and we shall agree that $$S_n(f) = \left(\sum_{i=0}^n|d_if|^2 \right)^{\frac{1}{2}}\quad\mbox{and}\quad s_n(f) = \left(\sum_{i=0}^n\mathbb{E}_{i-1}|d_if|^2 \right)^{\frac{1}{2}}.$$ Let $\rho$ be the space of all sequences $\varrho=(\varrho_n)_{n\ge0}$ of adapted (that is for all $n\in\mathbb{Z},\,\,\varrho_n$ is $\mathcal{F}_n$-measurable), non-decreasing, non-negative functions and define $$\varrho_{\infty} := \lim_{n\to\infty}\varrho_n.$$ We are now in the position to define the martingale Hardy-amalgam spaces. These spaces were originally introduce in \cite{BanSeh}. Let $0<p,q\le \infty$. Then 
\begin{itemize}
\item[i.]$H^S_{p,q}(\mathbb{R})$  is the space of all $f \in\mathcal{M}$ such that $S(f)\in L_{p,q}(\mathbb{R})$ with (quasi)-norm $$\|f\|_{H^S_{p,q}(\mathbb{R})} := \|S(f)\|_{L_{p,q}(\mathbb{R})}.$$
\item[ii.]\label{Hs}$H^s_{p,q}(\mathbb{R})$ is the space of all $f \in\mathcal{M}$ such that $s(f)\in L_{p,q}(\mathbb{R})$ with (quasi)-norm $$\|f\|_{H^s_{p,q}(\mathbb{R})} := \|s(f)\|_{L_{p,q}(\mathbb{R})}.$$
\item[iii.]$H^*_{p,q}(\mathbb{R})$ is the space of all $f \in\mathcal{M}$ such that $f^*\in L_{p,q}(\mathbb{R})$ with (quasi)-norm $$\|f\|_{H^*_{p,q}(\mathbb{R})} := \|f^*\|_{L_{p,q}(\mathbb{R})}.$$
\item[iv.]$\mathcal{Q}_{p,q}(\mathbb{R})$ is the space of all $f\in\mathcal{M}$ for which there is a
sequence of functions $\varrho=(\varrho_n)_{n\ge 0}\in \rho$ such that $S_n(f)\le\varrho_{n-1}$ and $\|\varrho_{\infty}\|_{L_{p,q}(\mathbb{R})}<\infty$ with (quasi)-norm $$\|f\|_{\mathcal{Q}_{p,q}(\mathbb{R})}:=\inf_{\varrho\in\rho}\|\varrho_{\infty}\|_{L_{p,q}(\mathbb{R})}.$$
\item[v.]$\mathcal{P}_{p,q}(\mathbb{R})$ is the space of all $f\in\mathcal{M}$ for which there is a
sequence of functions $\varrho=(\varrho_n)_{n\ge 0}\in \rho$ such that $|f_n|\le\varrho_{n-1}$ and $\|\varrho_{\infty}\|_{L_{p,q}(\mathbb{R})}<\infty$ with (quasi)-norm $$\|f\|_{\mathcal{P}_{p,q}(\mathbb{R})}:=\inf_{\varrho\in\rho}\|\varrho_{\infty}\|_{L_{p,q}(\mathbb{R})}.$$
\end{itemize}
We also introduce here the spaces $\mathcal{G}_{p,q}(\mathbb{R})$ and $\mathcal{BD}_{p,q}(\mathbb{R})$ which we shall refer to as variation integrable space and Jump bounded space respectively.  $$\mathcal{G}_{p,q}(\mathbb{R}):=\left\{f\in\mathcal{M}:\sum_{n=0}^{\infty} |d_nf| \in L_{p,q}(\mathbb{R})\right\}$$ endowed with the norm $$\|f\|_{\mathcal{G}_{p,q}(\mathbb{R})}=\left\|\sum_{n=0}^{\infty}|d_nf|\right\|_{L_{p,q}(\mathbb{R})}$$ for $1\le p\le q<\infty$ and $$\mathcal{BD}_{p,q}(\mathbb{R}):=\left\{f\in\mathcal{M}:\sup_{n\in\mathbb{N}}|d_nf|\in L_{p,q}(\mathbb{R})\right\}$$ endowed with the norm $$\|f\|_{\mathcal{BD}_{p,q}(\mathbb{R})}=\left\|\sup_{n\in\mathbb{N}}|d_nf|\right\|_{L_{p,q}(\mathbb{R})}$$ for $1\le p\le q\le\infty.$

For the sake of presentation, we sometimes write $H^S_{p,q}(\mathbb{R}),H^s_{p,q}(\mathbb{R}),H^*_{p,q}(\mathbb{R}),$ $\mathcal{Q}_{p,q}(\mathbb{R}),\mathcal{P}_{p,q}(\mathbb{R}), \mathcal{G}_{p,q}(\mathbb{R}), \mathcal{BD}_{p,q}(\mathbb{R})$ as $H^S_{p,q},H^s_{p,q},H^*_{p,q}$, $\mathcal{Q}_{p,q}, \mathcal{P}_{p,q}, \mathcal{G}_{p,q}, \mathcal{BD}_{p,q}$ respectively and $\|\cdot\|_{L_{p,q}(\mathbb{R})}$ as $\|\cdot\|_{p,q}.$
\vskip .1cm
We say that the stochastic basis, $\mathcal{F}_n,$ is regular if there exists $R>0$ such that $f_n\le Rf_{n-1}.$

\section{Presentation of Results}\label{sec3}
In this section, we present the main results of this paper. We start with the presentation of the inequalities that relate the first five spaces defined above. To be specific, we have the following theorem.
\begin{theorem}[Martingale Embeddings]\label{thm:mt1} Let $0<q\leq\infty$. Then
\begin{itemize}
\item[(i)] $\|f\|_{H^*_{p,q}(\mathbb{R})}\le C\|f\|_{H^s_{p,q}(\mathbb{R})},\quad \|f\|_{H^S_{p,q}(\mathbb{R})}\le C\|f\|_{H^s_{p,q}(\mathbb{R})}\qquad (0<p\le 2).$
\item[(ii)] $\|f\|_{H^s_{p,q}(\mathbb{R})}\le C\|f\|_{H^*_{p,q}(\mathbb{R})},\quad \|f\|_{H^s_{p,q}(\mathbb{R})}\le C\|f\|_{H^S_{p,q}(\mathbb{R})}\qquad (2\le p<\infty).$
\item[(iii)] $\|f\|_{H^*_{p,q}(\mathbb{R})}\le C\|f\|_{\mathcal{P}_{p,q}(\mathbb{R})},\quad \|f\|_{H^S_{p,q}(\mathbb{R})}\le C\|f\|_{\mathcal{Q}_{p,q}(\mathbb{R})}\qquad (0<p<\infty).$
\item[(iv)] $\|f\|_{H^*_{p,q}(\mathbb{R})}\le C\|f\|_{\mathcal{Q}_{p,q}(\mathbb{R})},\quad \|f\|_{H^S_{p,q}(\mathbb{R})}\le C\|f\|_{\mathcal{P}_{p,q}(\mathbb{R})}\qquad (0<p<\infty).$
\item[(v)] $\|f\|_{H^s_{p,q}(\mathbb{R})}\le C\|f\|_{\mathcal{P}_{p,q}(\mathbb{R})},\quad \|f\|_{H^s_{p,q}(\mathbb{R})}\le C\|f\|_{\mathcal{Q}_{p,q}(\mathbb{R})}\qquad (0<p<\infty).$
\end{itemize}
Moreover, if $(\mathcal{F}_n)_{n\geq 0}$ is regular, then $H^S_{p,q}(\mathbb{R}),H^s_{p,q}(\mathbb{R}),H^*_{p,q}(\mathbb{R})$,$\mathcal{Q}_{p,q}(\mathbb{R})$ and $\mathcal{P}_{p,q}(\mathbb{R})$ are all equivalent.
\end{theorem} 
A proof for the above theorem is provided in Section \ref{sec4} below. In this same Section \ref{sec4}, we prove the following extensions of Doob's inequality and Burkholder-Davis-Gundy inequality respectively.
\begin{theorem}\label{doobs}
Let $0<q\le\infty$ and $1<p<\infty.$ For every non-negative $L_{p,q}-$bounded submartingale $(f_n,\,\,n\in\mathbb{N}),$ we have that \begin{equation}\label{doob1}\left\|\sup_{n\in\mathbb{N}}f_n\right\|_{p,q}\le\frac{p}{p-1}\sup_{n\in\mathbb{N}}\|f_n\|_{p,q}.\end{equation}
\end{theorem}
\begin{theorem}\label{BDGamalgam}
The spaces $H_{p,q}^S(\mathbb{R})$ and $H_{p,q}^*(\mathbb{R})$ are equivalent for $1\le p,q\le\infty$, namely,
$$c_{p}\|f\|_{H_{p,q}^S(\mathbb{R})}\le \|f\|_{H_{p,q}^*(\mathbb{R})}\le C_{p}\|f\|_{H_{p,q}^S(\mathbb{R})}\quad (1\le p,q<\infty)$$
and 
$$c_{p}\|f\|_{H_{p,\infty}^S(\mathbb{R})}\le \|f\|_{H_{p,\infty}^*(\mathbb{R})}\le C_{p}\|f\|_{H_{p,\infty}^S(\mathbb{R})}\quad (1\le p<\infty).$$
\end{theorem}
In Section \ref{sec5}, we provide the proof of the following dual characterizations of the spaces $H^s_{p,q}(\mathbb{R})$ and $\mathcal{G}_{p,q}(\mathbb{R}).$ 
\begin{theorem}\label{thm:dualitygrand}
If either $1< q\le p\le 2$ or $2\le p\le q<\infty$, then the dual space of $H^s_{p,q}(\mathbb{R})$ identifies with $H^s_{p',q'}(\mathbb{R})$ where $\frac{1}{p}+\frac{1}{p'}=\frac{1}{q}+\frac{1}{q'}=1.$
\end{theorem}
\begin{theorem}\label{Garcia}
Let $1< p, q<\infty$. 
Then the dual space of $\mathcal{G}_{p,q}(\mathbb{R})$ is $\mathcal{BD}_{p',q'}(\mathbb{R})$ where $\frac{1}{p}+\frac{1}{p'}=1$ and $\frac{1}{q}+\frac{1}{q'}=1$
\end{theorem}

\section{Martingale Embeddings}\label{sec4}
In this section we will discuss the various inclusions of the martingale Hardy-amalgam spaces $H^S_{p,q}(\mathbb{R}),\,\, H^s_{p,q}(\mathbb{R}),\,\,H^*_{p,q}(\mathbb{R}),\,\,\mathcal{Q}_{p,q}(\mathbb{R})$ and $\mathcal{P}_{p,q}(\mathbb{R})$ as described in Theorem \ref{thm:mt1} above. Ferenc (\cite{Ferenc}) has discussed the classical cases including the Doob's maximal inequality for $p>1$ and the Burkholder-Davis-Gundy inequality. We recall here that the disjoint cover $(A_j)_{j\in \mathbb{Z}}$ is such that $A_j\in \mathcal{F}_n$ for all $j\in \mathbb{Z}$ and all $n\geq 1$. 

We refer to (\cite[Theorem 2.11]{Ferenc}) for the following classical result. 
\begin{proposition}\label{prop:weisz}
For any $f\in \mathcal{M}$, the following hold.
\begin{itemize}
\item[(i)] $\|f\|_{H^*_{p}(\mathbb{R})}\le C_p\|f\|_{H^s_{p}(\mathbb{R})},\quad \|f\|_{H^S_{p}(\mathbb{R})}\le C_p\|f\|_{H^s_{p}(\mathbb{R})}\qquad (0<p\le2)$
\item[(ii)] $\|f\|_{H^s_{p}(\mathbb{R})}\le C_p\|f\|_{H^*_{p}(\mathbb{R})},\quad \|f\|_{H^s_{p}}(\mathbb{R})\le C_p\|f\|_{H^S_{p}(\mathbb{R})}\qquad (2\le p<\infty)$
\item[(iii)] $\|f\|_{H^*_{p}(\mathbb{R})}\le C_p\|f\|_{\mathcal{P}_{p}(\mathbb{R})},\quad \|f\|_{H^S_{p}(\mathbb{R})}\le C_p\|f\|_{\mathcal{Q}_{p}(\mathbb{R})}\qquad (0<p<\infty)$
\item[(iv)] $\|f\|_{H^*_{p}(\mathbb{R})}\le C_p\|f\|_{\mathcal{Q}_{p}(\mathbb{R})},\quad \|f\|_{H^S_{p}(\mathbb{R})}\le C_p\|f\|_{\mathcal{P}_{p}(\mathbb{R})}\qquad (0<p<\infty)$
\item[(v)] $\|f\|_{H^s_{p}(\mathbb{R})}\le C_p\|f\|_{\mathcal{P}_{p}(\mathbb{R})},\quad \|f\|_{H^s_{p}(\mathbb{R})}\le C_p\|f\|_{\mathcal{Q}_{p}(\mathbb{R})}\qquad (0<p<\infty)$.
\end{itemize}
Moreover, if the $(\mathcal{F})_{n\ge 0}$ is regular, the above five spaces are equivalent.
\end{proposition}
We observe the following.
\begin{lemma}\label{lem:idmartingale}
Assume that $A\in \mathcal{F}_n$ for all $n\geq 1$. Then if $f\in \mathcal{M}$, then $f\mathbf{1}_A=(f_n\mathbf{1}_A)_{n\ge 0}$ is also  a martingale in $\mathcal{M}$. Moreover, if $T$ is any of the operators $s$, $S$ and $M$ (the maximal operator), then $$T(f\mathbf{1}_A)=T(f)\mathbf{1}_A.$$
\end{lemma}
Combining the above lemma with (\cite[Lemma 2.20]{Ferenc}), we obtain the following.
\begin{lemma}\label{lem:wesz220}
Assume that $A\in \mathcal{F}_n$ for all $n\geq 1$. Then for any martingale $f\in \mathcal{M}$ and $0<p<\infty$, we have
$$\mathbb{E}\left[\sup_n\mathbb{E}_{n-1}\left(|f_n|^p\mathbf{1}_A\right)\right]\le 2\mathbb{E}\left((f^*)^p\mathbf{1}_A\right)$$
and
$$\mathbb{E}\left[\sup_n\mathbb{E}_{n-1}\left(|S_n(f)^p\mathbf{1}_A\right)\right]\le 2\mathbb{E}\left(S(f)^p\mathbf{1}_A\right).$$
\end{lemma}

With the help of Lemma \ref{lem:idmartingale} and Lemma \ref{lem:wesz220} we now present the proof to the martingale inequalities that relate the five sets. 
\begin{proof}[Proof of Theorem \ref{thm:mt1}]
Let $T$ be any of the operators $s$, $S$ and $M$, and $H_{p}^T(\mathbb{R}),H_{p,q}^T(\mathbb{R})$ the corresponding martingale spaces, then since $A_j\in \mathcal{F}_n$ for all $j\in \mathbb{Z}$ and all $n\geq 1$, by Lemma \ref{lem:idmartingale}, we have that for any martingale $f\in\mathcal{M}$ and any $j\in \mathbb{Z}$,
$$\int_\mathbb{R} T(f)^p\mathbf{1}_{A_j}d\mathbb{P}=\int_\mathbb{R} T(f\mathbf{1}_{A_j})^pd\mathbb{P}=\|f\mathbf{1}_{A_j}\|_{H_{p}^T(\mathbb{R})}^p,$$
and consequently,
\begin{equation}\label{eq:backtoweisz}\|T(f)\|_{H_{p,q}^T(\mathbb{R})}^q=\sum_{j}\|T(f)\mathbf{1}_{A_j}\|_{L_p(\mathbb{R})}^q=\sum_{j}\|f\mathbf{1}_{A_j}\|_{H_{p}^T(\mathbb{R})}^q.\end{equation}
The two first assertions of the theorem then follow from (\ref{eq:backtoweisz}) and Proposition \ref{prop:weisz}.
\vskip .1cm
To obtain the other assertions, following (\ref{eq:backtoweisz}) and Proposition \ref{prop:weisz}, we only need to prove that 
\begin{equation}\label{eq:backtoweisz1}\sum_{j}\|f\mathbf{1}_{A_j}\|_{\mathcal{Q}_{p}(\mathbb{R})}^q\le C\|f\|_{\mathcal{Q}_{p,q}(\mathbb{R})}^q\end{equation}
and
\begin{equation}\label{eq:backtoweisz2}\sum_{j}\|f\mathbf{1}_{A_j}\|_{\mathcal{P}_{p}(\mathbb{R})}^q\le C\|f\|_{\mathcal{P}_{p,q}(\mathbb{R})}^q.\end{equation}
We only prove (\ref{eq:backtoweisz1}) as (\ref{eq:backtoweisz2}) follows similarly.

Let $(\varrho_n)_{n\ge 0}$ be an arbitrary nonnegative nondecreasing adapted sequence such that
$$S_n(f)\le \varrho_{n-1},\,\,\textrm{and}\,\,\|\varrho_\infty\|_{L_{p,q}(\mathbb{R})}<\infty.$$
We have that the sequence $(\gamma_n^j)_{n\ge 0}=(\varrho_n\mathbf{1}_{A_j})_{n\ge 0}$ is also nonnegative nondecreasing and adapted, and  
$$S_n(f\mathbf{1}_{A_j})=S_n(f)\mathbf{1}_{A_j}\le \varrho_{n-1}\mathbf{1}_{A_j}=\gamma_{n-1}^j$$ and $$\|\gamma_\infty^j\|_{L_p(\mathbb{R})}=\|\varrho_\infty \mathbf{1}_{A_j}\|_{L_p(\mathbb{R})}\le\|\varrho_\infty\|_{L_{p,q}(\mathbb{R})} <\infty.$$
It follows that
$$\sum_{j}\|f\mathbf{1}_{A_j}\|_{\mathcal{Q}_{p}(\mathbb{R})}^q\le \sum_{j}\|\gamma_\infty^j\|_{L_p(\mathbb{R})}^q=\sum_{j}\|\varrho_\infty \mathbf{1}_{A_j}\|_{L_p(\mathbb{R})}^q=\|\varrho_\infty\|_{L_{p,q}(\mathbb{R})}^q.$$
As the sequence $(\varrho_n)_{n\ge 0}$ was chosen arbitrarily, we conclude that
$$\sum_{j}\|f\mathbf{1}_{A_j}\|_{\mathcal{Q}_{p}(\mathbb{R})}^q\le \inf_{\varrho\in\rho}\|\varrho_\infty\|_{L_{p,q}(\mathbb{R})}^q=\|f\|_{\mathcal{Q}_{p,q}(\mathbb{R})}^q.$$
Let us now assume that the $(\mathcal{F}_n)_{n\ge 0}$ is regular. To prove the equivalence between the five spaces, we only need to prove that
\begin{equation}\label{eq:embedequiv1}\|f\|_{\mathcal{Q}_{p,q}(\mathbb{R})}\le C\|f\|_{H_{p,q}^S(\mathbb{R})}\end{equation}
and \begin{equation}\label{eq:embedequiv2}\|f\|_{\mathcal{P}_{p,q}(\mathbb{R})}\le C\|f\|_{H_{p,q}^*(\mathbb{R})}.\end{equation}
We only prove the (\ref{eq:embedequiv1}) since the proof of (\ref{eq:embedequiv2}) use similar arguments.
\vskip .1cm
Let $f=(f_n)_{n\ge 0}$ be a martingale in $H_{p,q}^S(\mathbb{R})$. Then using the definition of the regularity, one obtain that
\begin{equation}\label{eq:Sineq}S_n(f)\le \left[C_p\left(S_{n-1}^p(f)+\mathbb{E}_{n-1}(S_n^p(f))\right)\right]^{\frac 1p}\end{equation}
(see \cite[p. 39]{Ferenc}).
Define the sequence $\varrho=(\varrho_n)_{n\ge 0}$ by $$\varrho_n=\left[C_p\left(S_{n}^p(f)+\mathbb{E}_{n}(S_{n+1}^p(f))\right)\right]^{\frac 1p}.$$
Then $\varrho\in \rho$ and by (\ref{eq:Sineq}), $$S_n(f)\le \varrho_{n-1}.$$
Also, we have that
$$\varrho_\infty=\sup_n\varrho_n=\left[C_p\left(S^p(f)+\sup_n\mathbb{E}_{n}(S_{n+1}^p(f))\right)\right]^{\frac 1p}.$$ 
Then using Lemma \ref{lem:idmartingale} and Lemma \ref{lem:wesz220}, we obtain for any $j\in \mathbb{Z}$,
$$\|\varrho_\infty\mathbf{1}_{A_j}\|_{L_p(\mathbb{R})}\le 3C_p\|S^p(f)\mathbf{1}_{A_j}\|_{L_p(\mathbb{R})}.$$
Hence $$\|f\|_{\mathcal{Q}_{p,q}(\mathbb{R})}\le \|\varrho_\infty\|_{L_{p,q}(\mathbb{R})}\lesssim \|S(f)\|_{L_{p,q}(\mathbb{R})}=\|f\|_{H_{p,q}^S(\mathbb{R})}.$$
The proof is complete.
\end{proof}

We finish this section with the proofs of martingale inequalities, the generalization of Doob's inequality and Burkholder-Davis-Gundy inequality. For this, we recall the following classical Doob's inequality, that can be found in \cite{Ferenc}.
\begin{proposition}\label{cl1}
Let $p>1.$ For every non-negative $L_p-$bounded submartingale $(f_n,\,\,n\in\mathbb{N}),$ we have that $$\left\|\sup_{n\in\mathbb{N}}f_n\right\|_{L_p(\mathbb{R})}\le\frac{p}{p-1}\sup_{n\in\mathbb{N}}\|f_n\|_{L_p(\mathbb{R})}.$$
\end{proposition}
We start with the proof of the extension of Doob's inequality.
\begin{proof}[Proof of Theorem \ref{doobs}]
Let $A_j$ be defined as equation (\ref{dia1}). Let $g_{n,j}=f_n\mathbf{1}_{A_j}.$ Since $g_{n,j}$ is a martingale, Proposition \ref{cl1} implies that $$\left\|\sup_{n\in\mathbb{N}}g_{n,j}\right\|_{L_p(\mathbb{R})}\le\frac{p}{p-1}\sup_{n\in\mathbb{N}}\|g_{n,j}\|_{L_p(\mathbb{R})}.$$ Therefore by definition, and for $0< q<\infty$,
\begin{eqnarray*}
\|\sup_{n\in\mathbb{N}}f_n\|_{L_{p,q}(\mathbb{R})}^q &=& \sum_j\left\|\sup_{n\in\mathbb{N}}f_n\mathbf{1}_{A_j}\right\|_p^q=\sum_j\left\|\sup_{n\in\mathbb{N}}g_{n,j}\right\|_p^q\\
&\le&\left(\frac{p}{p-1}\right)^q\sum_j\left(\sup_{n\in\mathbb{N}}\|g_{n,j}\|_{p}\right)^q\\
&= & \left(\frac{p}{p-1}\right)^q\sum_j\sup_{n\in\mathbb{N}}\|g_{n,j}\|_p^q = \left(\frac{p}{p-1}\right)^q\sup_{n\in\mathbb{N}}\sum_j\|g_{n,j}\|_p^q\\
& = &\left(\frac{p}{p-1}\right)^q\sup_{n\in\mathbb{N}}\sum_j\|f_n\mathbf{1}_{A_j}\|_p^q
\end{eqnarray*}
Thus $$\|\sup_{n\in\mathbb{N}}f_n\|_{L_{p,q}(\mathbb{R})}^q\le \left(\frac{p}{p-1}\right)^q\sup_{n\in\mathbb{N}}\sum_j\|f_n\mathbf{1}_{A_j}\|_p^q=\left(\frac{p}{p-1}\right)^q\sup_{n\in\mathbb{N}}\|f\|_{L_{p,q}(\mathbb{R})}^q.$$ The case $q=\infty$ follows similarly. The proof is complete. 
\end{proof}
We also obtain the proof of the extension of Burkholder-Davis-Gundy's inequality below.
\begin{proof}[Proof of Theorem \ref{BDGamalgam}]
The proof follows from Lemma \ref{lem:idmartingale} and the classical Burkholder-Davis-Gundy's inequality (see \cite[Theorem 2.12]{Ferenc}). We note that for the second equivalence, we simply replace the summation with the supremum and the result follows.
\end{proof}
 
\section{Dual Characterizations}\label{sec5}
In this Section, we focus our attention on the characterization of the the dual of the spaces $H^s_{p,q}(\mathbb{R})$ when $1<p,q<\infty$ and $\mathcal{G}_{p,q}(\mathbb{R}).$ Let us start by identifying the dual space of $H^s_{p,q}(\mathbb{R}).$

\subsection{Dual of $H^s_{p,q}(\mathbb{R})$}
As noted in the introduction, the dual of $H^s_{p,q}(\mathbb{R})$ when $0<p\le q\le1$ is a Campanato-type space (see \cite{BanSeh}).  Hence our focus is on characterizing the dual of $H^s_{p,q}(\mathbb{R})$ when $1<p\le q<\infty$. Let us start with the justification of the following important result.
\begin{lemma}\label{lem:uniconvex}
Let $2\le p\le q<\infty$. Then the space $H_{p,q}^s(\mathbb{R})$ is uniformly convex.
\end{lemma}
\begin{proof}
We recall that a Banach space $\mathcal{H}$ is uniformly convex if for any $\epsilon > 0$, there exists $\delta >0$ such that if $x,y\in \mathcal{H}$ with $\|x\|_{\mathcal{H}}\le 1$, $\|y\|_{\mathcal{H}}\le 1$ and $\|x-y\|_{\mathcal{H}}\ge \epsilon$, then $\|x+y\|_{\mathcal{H}}\le 2(1-\delta)$.
\vskip .1cm
We recall that for $1\le r<\infty$ and for $a,b>0$, $$(a+b)^r\le 2^{r-1}(a^r+b^r)\quad\textrm{and}\quad a^r+b^r\le (a+b)^r.$$
Let $\epsilon>0$, and assume that $f,g\in H_{p,q}^s$ with $\|f\|_{H_{p,q}^s}\le 1$, $\|g\|_{H_{p,q}^s}\le 1$ and $\|f-g\|_{H_{p,q}^s}\ge \epsilon$. We start by observing that $$s^2(f+g)+s^2(f-g)=2(s^2(f)+s^2(g)).$$
We then obtain
\begin{eqnarray*}
\left(s^2(f+g)\right)^{\frac p2}+\left(s^2(f-g)\right)^{\frac p2}\le \left(s^2(f+g)+s^2(f-g)\right)^{\frac p2}\le 2^{p-1}\left[s^p(f)+s^p(g)\right].
\end{eqnarray*}
Hence for any $j\in \mathbb{Z}$,
$$\|s(f+g)\mathbf{1}_{A_j}\|_p^p+\|s(f+g)\mathbf{1}_{A_j}\|_p^p\le 2^{p-1}\left(\|s(f)\mathbf{1}_{A_j}\|_p^p+\|s(g)\mathbf{1}_{A_j}\|_p^p\right).$$
Raising both members of the last inequality to the power $\frac qp\ge 1$, we obtain
\begin{eqnarray*}
\|s(f+g)\mathbf{1}_{A_j}\|_p^q+\|s(f+g)\mathbf{1}_{A_j}\|_p^q &\le& \left(\|s(f+g)\mathbf{1}_{A_j}\|_p^p+\|s(f+g)\mathbf{1}_{A_j}\|_p^p\right)^{\frac qp}\\ &\le& 2^{\frac{q}{p}(p-1)}\left(\|s(f)\mathbf{1}_{A_j}\|_p^p+\|s(g)\mathbf{1}_{A_j}\|_p^p\right)^{\frac qp}\\ &\le& 2^{\frac{q}{p}(p-1)}2^{\frac qp-1}\left(\|s(f)\mathbf{1}_{A_j}\|_p^q+\|s(g)\mathbf{1}_{A_j}\|_p^q\right).
\end{eqnarray*}
Hence taking the sum over $j\in \mathbb{Z}$, we obtain
$$\|s(f+g)\|_{p,q}^q+\|s(f-g)\|_{p,q}^q\le 2^{q-1}\left(\|s(f)\|_{p,q}^q+\|s(g)\|_{p,q}^q\right)$$
and so 
\begin{eqnarray*}\|s(f+g)\|_{p,q}^q &\le& 2^{q-1}\left(\|s(f)\|_{p,q}^q+\|s(g)\|_{p,q}^q\right)-\|s(f-g)\|_{p,q}^q\\ &\le& 2^{q}-\epsilon^q.
\end{eqnarray*}
Thus $$\|s(f+g)\|_{p,q}\le 2(1-\delta)$$ where $$\delta=1-\left(1-\frac{\epsilon^q}{2^q}\right)^{\frac 1q}$$ and the proof is complete.
\end{proof}
From the above Lemma and Milman's Theorem (see \cite[p.127]{Yosida}), we deduce the following.
\begin{corollary}\label{cor:Milman}
Let $2\le p\le q<\infty$. Then the space $H_{p,q}^s(\mathbb{R})$ is reflexive.
\end{corollary}
\begin{proof}[Proof of Theorem \ref{thm:dualitygrand}]
It follows from Corollary \ref{cor:Milman} above that we only need to prove for the case $1< q\le p\le 2.$ Let $g\in H^s_{p',q'}(\mathbb{R})$ and   \begin{eqnarray*}\label{dk1}\kappa_g(f):=\mathbb{E}\left(\sum_{n=0}^{\infty}d_nfd_ng\right)\quad \left(f\in H^s_{p,q}(\mathbb{R})\right).\end{eqnarray*} 
Hence by Schwarz's inequality, we have that
\begin{eqnarray*}
|\kappa_g(f)| &\le& \int_{\mathbb{R}}\sum_{n=0}^{\infty}\mathbb{E}_{n-1}|d_nf||d_ng|\mathrm{d}\mathbb{P}\\
&=& \sum_{j\in\mathbb{Z}}\int_{A_j}\sum_{n=0}^{\infty}\mathbb{E}_{n-1}|d_nf||d_ng|\mathrm{d}\mathbb{P}\\
&\le& \sum_{j\in\mathbb{Z}}\int_{A_j}\sum_{n=0}^{\infty}\left(\mathbb{E}_{n-1}|d_nf|^2\right)^{\frac{1}{2}}\left(\mathbb{E}_{n-1}|d_ng|^2\right)^{\frac{1}{2}}\mathrm{d}\mathbb{P}\\
&\le&\sum_{j\in\mathbb{Z}}\int_{A_j}\left(\sum_{n=0}^{\infty}\mathbb{E}_{n-1}|d_nf|^2\right)^{\frac{1}{2}}\left(\sum_{n=0}^{\infty}\mathbb{E}_{n-1}|d_ng|^2\right)^{\frac{1}{2}}\mathrm{d}\mathbb{P}\\
&=&\sum_{j\in\mathbb{Z}}\int_{A_j}s(f)s(g)\mathrm{d}\mathbb{P}.
\end{eqnarray*}
Applying the H\"older's inequality to the right hand of the last inequality, we obtain 
\begin{eqnarray*}
|\kappa_g(f)| &\le& \sum_{j\in\mathbb{Z}}\|s(f)\mathbf{1}_{A_j}\|_p\|s(g)\mathbf{1}_{A_j}\|_{p'}\\
&\le&\left(\sum_{j\in\mathbb{Z}}\|s(f)\mathbf{1}_{A_j}\|_p^q\right)^{\frac{1}{q}}\left(\sum_{j\in\mathbb{Z}}\|s(g)\mathbf{1}_{A_j}\|_{p'}^{q'}\right)^{\frac{1}{q'}}\\
&=&\|f\|_{H^s_{p,q}(\mathbb{R})}\|g\|_{H^s_{p',q'}(\mathbb{R})}.
\end{eqnarray*}
Thus $\kappa_g\in \left(H^s_{p,q}\right)'$ and $$\|\kappa_g\|\le \|g\|_{H^s_{p',q'}(\mathbb{R})}.$$

Conversely, let $\kappa$ be a continuous linear functional on $H^s_{p,q}(\mathbb{R}).$ Then as $H^s_{p,q}(\mathbb{R})$ embeds continuously into $H^s_{p}(\mathbb{R})$ (since $q<p$), we have by the Hahn-Banach theorem that $\kappa$ can be extended to a continuous linear functional $\tilde{\kappa}$ on $H^s_{p}(\mathbb{R})$ having the same operator norm as $\kappa$. It follows from \cite[Theorem 2.26]{Ferenc} that there exists some $g\in H^s_{p'}(\mathbb{R})$ such that  $$\tilde{\kappa}(f) = \mathbb{E}(fg)\quad \left(\forall f\in H^s_{p}(\mathbb{R})\right).$$ 
In particular
\begin{equation}\label{eq:dualityrep}\kappa(f)=\tilde{\kappa}(f) = \mathbb{E}(fg)\quad \left(\forall f\in H^s_{p,q}(\mathbb{R})\right).\end{equation}
Let us prove that \begin{equation}\label{eq:dualnecess}\|g\|_{H^s_{p',q'}(\mathbb{R})}\lesssim \sup_{f\in H^s_{p,q}(\mathbb{R}),\,\|f\|_{H^s_{p,q}(\mathbb{R})}\le 1}|\kappa(f)|<\infty.\end{equation}
Obviously, this holds if $\|g\|_{H^s_{p',q'}(\mathbb{R})}=0$. Hence we assume that $\|g\|_{H^s_{p',q'}(\mathbb{R})}\neq 0$. 
\vskip .1cm
We recall that,  $A_j\in\mathcal{F}_n$ for all $j\in \mathbb{Z}$ and $n\ge 1$. Set  \begin{eqnarray}\label{trans1}\mu_n = \sum_{j\in\mathbb{Z}}\frac{s^{p'-2}_n(g)\mathbf{1}_{A_j}}{\|s(g)\|_{p',q'}^{q'-1}\|s(g)\mathbf{1}_{A_j}\|_{p'}^{p'-q'}}.\end{eqnarray} Since the $A_j$'s are pairwise disjoint, we have that $$\mu_n^2 = \sum_{j\in\mathbb{Z}}\frac{s^{2p'-4}_n(g)\mathbf{1}_{A_j}}{\|s(g)\|_{p',q'}^{2q'-2}\|s(g)\mathbf{1}_{A_j}\|_{p'}^{2(p'-q')}}.$$ From the definition of $s(\cdot),$ we have that $\mu_n$ is $\mathcal{F}_{n-1}$-measurable. We define $h$ as the martingale transform of $g$ by $\mu_n.$ That is \begin{eqnarray}\label{trans2}d_nh=\mu_n d_ng.\end{eqnarray} We then obtain $$\sum_{n=0}^{\infty}\mathbb{E}_{n-1}|d_nh|^2=\sum_{n=0}^{\infty}\mu_n^2\mathbb{E}_{n-1}|d_ng|^2$$ or equivalently $$s^2(h)=\sum_{n=0}^{\infty}\sum_{j\in\mathbb{Z}}\frac{s^{2p'-4}_n(g)\mathbf{1}_{A_j}}{\|s(g)\|_{p',q'}^{2q'-2}\|s(g)\mathbf{1}_{A_j}\|_{p'}^{2(p'-q')}}\mathbb{E}_{n-1}|d_ng|^2.$$ Therefore
\begin{eqnarray*}
s^2(h) &=& \sum_{j\in\mathbb{Z}}\frac{\mathbf{1}_{A_j}}{\|s(g)\|_{p',q'}^{2q'-2}\|s(g)\mathbf{1}_{A_j}\|_{p'}^{2(p'-q')}}\sum_{n=0}^{\infty}s^{2p'-4}_n(g)\mathbb{E}_{n-1}|d_ng|^2\\
&=& \sum_{j\in\mathbb{Z}}\frac{\mathbf{1}_{A_j}}{\|s(g)\|_{p',q'}^{2q'-2}\|s(g)\mathbf{1}_{A_j}\|_{p'}^{2(p'-q')}}\sum_{n=0}^{\infty}s^{2p'-4}_n(g)(s_n^2(g)-s_{n-1}^2(g))\\
&=& \frac{1}{\|s(g)\|_{p',q'}^{2q'-2}}\sum_{j\in\mathbb{Z}}\frac{\mathbf{1}_{A_j}}{\|s(g)\mathbf{1}_{A_j}\|_{p'}^{2(p'-q')}}\sum_{n=0}^{\infty}[s^{2p'-2}_n(g)-s^{2p'-4}_n(g)s_{n-1}^2(g)].
\end{eqnarray*}
It follows that
\begin{eqnarray*}
s^2(h)&\le& \frac{1}{\|s(g)\|_{p',q'}^{2q'-2}}\sum_{j\in\mathbb{Z}}\frac{\mathbf{1}_{A_j}}{\|s(g)\mathbf{1}_{A_j}\|_{p'}^{2(p'-q')}}\sum_{n=0}^{\infty}[s^{2p'-2}_n(g)-s^{2p'-2}_{n-1}(g)]\\
&=& \frac{1}{\|s(g)\|_{p',q'}^{2q'-2}}\sum_{j\in\mathbb{Z}}\frac{s^{2p'-2}(g)\mathbf{1}_{A_j}}{\|s(g)\mathbf{1}_{A_j}\|_{p'}^{2(p'-q')}}.
\end{eqnarray*}
Thus, by disjointness of the $A_j$'s,
\begin{eqnarray}\label{trans3}
s(h) \le \frac{s^{p'-1}(g)}{\|s(g)\|_{p',q'}^{q'-1}}\sum_{j\in\mathbb{Z}}\frac{\mathbf{1}_{A_j}}{\|s(g)\mathbf{1}_{A_j}\|_{p'}^{p'-q'}}.
\end{eqnarray}

We also have that for any $k\in \mathbb{Z}$, $$s(h)\mathbf{1}_{A_k} \le  \sum_{j\in\mathbb{Z}}\frac{s^{p'-1}(g)}{\|s(g)\|_{p',q'}^{q'-1}}\frac{\mathbf{1}_{A_j}}{\|s(g)\mathbf{1}_{A_j}\|_{p'}^{p'-q'}}\mathbf{1}_{A_k}=\frac{s^{p'-1}(g)}{\|s(g)\|_{p',q'}^{q'-1}}\frac{\mathbf{1}_{A_k}}{\|s(g)\mathbf{1}_{A_k}\|_{p'}^{p'-q'}}.$$ Therefore $$\|s(h)\mathbf{1}_{A_k}\|_p\le\frac{\|s^{p'-1}(g)\mathbf{1}_{A_k}\|_p}{\|s(g)\|_{p',q'}^{q'-1}\|s(g)\mathbf{1}_{A_k}\|_{p'}^{p'-q'}}=\frac{\|s(g)\mathbf{1}_{A_k}\|_{p'}^{p'-1}}{\|s(g)\|_{p',q'}^{q'-1}\|s(g)\mathbf{1}_{A_k}\|_{p'}^{p'-q'}}=\frac{\|s(g)\mathbf{1}_{A_k}\|_{p'}^{q'-1}}{\|s(g)\|_{p',q'}^{q'-1}}.$$ Hence $$\sum_{k\in\mathbb{Z}}\|s(h)\mathbf{1}_{A_k}\|_p^q\le\sum_{k\in\mathbb{Z}}\frac{\|s(g)\mathbf{1}_{A_k}\|_{p'}^{q(q'-1)}}{\|s(g)\|_{p',q'}^{q(q'-1)}}=\sum_{k\in\mathbb{Z}}\frac{\|s(g)\mathbf{1}_{A_k}\|_{p'}^{q'}}{\|s(g)\|_{p',q'}^{q'}}=\frac{\|s(g)\|_{p',q'}^{q'}}{\|s(g)\|_{p',q'}^{q'}}=1.$$ That is $$\|h\|_{H^s_{p,q}(\mathbb{R})}\le 1.$$    
We now test (\ref{eq:dualnecess}) with the martingale $h$ above. First proceeding as in \cite[p.37]{Ferenc} (this is why we need $p$ to be smaller than $2$), we obtain  
\begin{eqnarray*}
|\kappa(h)| &=& \mathbb{E}\left(\sum_{n=0}^{\infty}d_nhd_ng\right)=\mathbb{E}\left(\sum_{n=0}^{\infty}\mu_n|d_ng|^2\right)\\
&=& \frac{1}{\|s(g)\|_{p',q'}^{q'-1}}\mathbb{E}\left(\sum_{n=0}^{\infty}\sum_{j\in\mathbb{Z}}\frac{s^{p'-2}_n(g)\mathbf{1}_{A_j}}{\|s(g)\mathbf{1}_{A_j}\|_{p'}^{p'-q'}}\mathbb{E}_{n-1}|d_ng|^2\right)\\
&=& \frac{1}{\|s(g)\|_{p',q'}^{q'-1}}\mathbb{E}\left(\sum_{n=0}^{\infty}\sum_{j\in\mathbb{Z}}\frac{s^{p'-2}_n(g)\mathbf{1}_{A_j}}{\|s(g)\mathbf{1}_{A_j}\|_{p'}^{p'-q'}}(s^2_n(g)-s^2_{n-1}(g))\right)
\end{eqnarray*}
It follows that
\begin{eqnarray*}
|\kappa(h)|
&\ge& \frac{2}{p'}\frac{1}{\|s(g)\|_{p',q'}^{q'-1}}\sum_{j\in\mathbb{Z}}\frac{1}{\|s(g)\mathbf{1}_{A_j}\|_{p'}^{p'-q'}}\mathbb{E}\left(\mathbf{1}_{A_j}\sum_{n=0}^{\infty}s^{p'}_n(g)-s^{p'}_{n-1}(g)\right)\\
&=& \frac{2}{p'}\frac{1}{\|s(g)\|_{p',q'}^{q'-1}}\sum_{j\in\mathbb{Z}}\frac{1}{\|s(g)\mathbf{1}_{A_j}\|_{p'}^{p'-q'}}\mathbb{E}\left(\mathbf{1}_{A_j}s^{p'}(g)\right) \\
&=& \frac{2}{p'}\frac{1}{\|s(g)\|_{p',q'}^{q'-1}}\sum_{j\in\mathbb{Z}}\frac{1}{\|s(g)\mathbf{1}_{A_j}\|_{p'}^{p'-q'}}\int_{\mathbb{R}}\mathbf{1}_{A_j}s^{p'}(g)\mathrm{d}\mathbb{P}\\
&=&\frac{2}{p'}\frac{1}{\|s(g)\|_{p',q'}^{q'-1}}\sum_{j\in\mathbb{Z}}\frac{1}{\|s(g)\mathbf{1}_{A_j}\|_{p'}^{p'-q'}}\|s(g)\mathbf{1}_{A_j}\|_{p'}^{p'}\\
&=& \frac{2}{p'}\frac{1}{\|s(g)\|_{p',q'}^{q'-1}}\sum_{j\in\mathbb{Z}}\|s(g)\mathbf{1}_{A_j}\|_{p'}^{q'}\\
&=& \frac{2}{p'}\frac{\|s(g)\|_{p',q'}^{q'}}{\|s(g)\|_{p',q'}^{q'-1}}=\frac{2}{p'}\|s(g)\|_{L_{p',q'}(\mathbb{R})}.
\end{eqnarray*}
The proof is complete.
\end{proof}
\subsection{Dual of $\mathcal{G}_{p,q}$}
This part is devoted to the characterization of the dual of the variation integrable space. We begin our characterization of the dual of $\mathcal{G}_{p,q}$ with the introduction of the following larger space.
\begin{definition}
Let $n\in\mathbb{N}_0$ and let $1\le p,q,r<\infty.$ We define the space $\mathcal{K}(L_{p,q},\ell_r)$ by $$\mathcal{K}(L_{p,q},\ell_r)=\left\{\mbox{measurable process}\,\, \epsilon=(\epsilon_n)_{n\ge0}:\|\epsilon\|_{\mathcal{K}(L_{p,q},\ell_r)}<\infty\right\}$$ where $$\|\epsilon\|_{\mathcal{K}(L_{p,q},\ell_r)}=\left\|\left(\sum_{n\ge0}|\epsilon_n|^r\right)^{\frac{1}{r}}\right\|_{L_{p,q}(\mathbb{R})}.$$
\end{definition}
We observe that $\mathcal{G}_{p,q}\subseteq\mathcal{K}(L_{p,q},\ell_1).$ Indeed let $f$ be a martingale. Then it is measurable with respect to the underlining filtration hence its increment, $d_nf,$ is also measurable. Thus we can take $\epsilon_n=d_nf$ and the inclusion then follows by setting $r=1.$ 
In the same way, we obtain that $\mathcal{BD}_{p,q}\subseteq\mathcal{K}(L_{p,q},\ell_{\infty}).$ 

We also observe that since $L_{p,p}(\mathbb{R})=L_p(\mathbb{R}),$ then $\mathcal{K}(L_{p,p},\ell_r)$ is the space defined in \cite[Definition 2.8]{Ferenc}. 
The following lemma is part of the proof of Proposition \ref{pro1} that follows, but for the sake of the presentation, we isolate it.
\begin{lemma}\label{lemma1}
Let $1< p,q<\infty$, $1\le r<\infty$ and let $(p,p'),\,\,(q,q'),\,\,(r,r')$ be their respective conjugate pairs. Let $\eta\in\mathcal{K}(L_{p',q'},\ell_{r'}).$ Consider the sequence $h=(h_k)_{k\ge0}$ defined as follows $$h_k=\left\{\begin{array}{lcr}\sum_{i\ge0}\frac{|\eta_k|^{r'}}{\eta_k}\frac{\|\eta\|_{\ell_{r'}}^{p'-r'}}{\|\eta\|_{\mathcal{K}(L_{p',q'},\ell_{r'})}^{q'-1}}\frac{\mathbf{1}_{A_i}}{\left\|\|\eta\|_{\ell_{r'}}\mathbf{1}_{A_i}\right\|_{L_{p'}}^{p'-q'}}&,&\eta_k\ne0\\0&,&\mbox{otherwise}\end{array}\right.$$  if $r>1$, and 
$$h_k=\left\{\begin{array}{lcr}\sum_{i\ge0}\frac{\sign(\eta_k)}{2^{k+1}}\frac{\|\eta\|_{\ell_{r'}}^{p'-1}}{\|\eta\|_{\mathcal{K}(L_{p',q'},\ell_{r'})}^{q'-1}}\frac{\mathbf{1}_{A_i}}{\left\|\|\eta\|_{\ell_{r'}}\mathbf{1}_{A_i}\right\|_{L_{p'}}^{p'-q'}}&,&\eta_k\ne0\\0&,&\mbox{otherwise}\end{array}\right.$$ for $r=1$.
Then
 $h$  has a unit norm in $\mathcal{K}(L_{p,q},\ell_{r}).$ Consequently $h\in\mathcal{K}(L_{p,q},\ell_{r}).$
\end{lemma}
\begin{proof}
 By definition, 
 \begin{eqnarray*}
 \|h\|_{\mathcal{K}(L_{p,q},\ell_r)} &=& \left(\sum_{j\ge0}\left(\int_{\mathbb{R}}\left(\sum_{k}|h_k|^r\right)^{\frac{p}{r}}\mathbf{1}_{A_j}\mathrm{d}\mathbb{P}\right)^{\frac{q}{p}}\right)^{\frac{1}{q}}\\
 &=&\left(\sum_{j\ge0}\left\|\|h\|_{\ell_r}\mathbf{1}_{A_j}\right\|^q_{L_p}\right)^{\frac{1}{q}}.
 \end{eqnarray*}
 Now for $r>1$, we obtain $$|h_k|^r=|\eta_k|^{r'}\frac{\|\eta\|_{\ell_{r'}}^{(p'-r')r}}{\|\eta\|_{\mathcal{K}(L_{p',q'},\ell_{r'})}^{(q'-1)r}}\left(\sum_{i\ge0}\frac{\mathbf{1}_{A_i}}{\left\|\|\eta\|_{\ell_{r'}}\mathbf{1}_{A_i}\right\|_{L_{p'}}^{p'-q'}}\right)^r$$ so that 
 $$\sum_k|h_k|^r = \sum_k|\eta_k|^{r'}\frac{\|\eta\|_{\ell_{r'}}^{(p'-r')r}}{\|\eta\|_{\mathcal{K}(L_{p',q'},\ell_{r'})}^{(q'-1)r}}\left(\sum_{i\ge0}\frac{\mathbf{1}_{A_i}}{\left\|\|\eta\|_{\ell_{r'}}\mathbf{1}_{A_i}\right\|_{L_{p'}}^{p'-q'}}\right)^r$$ and hence 
\begin{eqnarray*} 
\|h\|_{\ell_r}^r &=& \|\eta\|_{\ell_{r'}}^{r'}\frac{\|\eta\|_{\ell_{r'}}^{(p'-r')r}}{\|\eta\|_{\mathcal{K}(L_{p',q'},\ell_{r'})}^{(q'-1)r}}\left(\sum_{i\ge0}\frac{\mathbf{1}_{A_i}}{\left\|\|\eta\|_{\ell_{r'}}\mathbf{1}_{A_i}\right\|_{L_{p'}}^{p'-q'}}\right)^r\\
&=& \frac{\|\eta\|_{\ell_{r'}}^{r'+(p'-r')r}}{\|\eta\|_{\mathcal{K}(L_{p',q'},\ell_{r'})}^{(q'-1)r}}\left(\sum_{i\ge0}\frac{\mathbf{1}_{A_i}}{\left\|\|\eta\|_{\ell_{r'}}\mathbf{1}_{A_i}\right\|_{L_{p'}}^{p'-q'}}\right)^r\\
&=& \frac{\|\eta\|_{\ell_{r'}}^{(p'-1)r}}{\|\eta\|_{\mathcal{K}(L_{p',q'},\ell_{r'})}^{(q'-1)r}}\left(\sum_{i\ge0}\frac{\mathbf{1}_{A_i}}{\left\|\|\eta\|_{\ell_{r'}}\mathbf{1}_{A_i}\right\|_{L_{p'}}^{p'-q'}}\right)^r
\end{eqnarray*}
and then \begin{equation}\label{eq:testestim}\|h\|_{\ell_r} = \frac{\|\eta\|_{\ell_{r'}}^{p'-1}}{\|\eta\|_{\mathcal{K}(L_{p',q'},\ell_{r'})}^{q'-1}}\sum_{i\ge0}\frac{\mathbf{1}_{A_i}}{\left\|\|\eta\|_{\ell_{r'}}\mathbf{1}_{A_i}\right\|_{L_{p'}}^{p'-q'}}.\end{equation}
One can easily check that (\ref{eq:testestim}) also holds for $r=1$.
\vskip .2cm
As the $A_k$s are disjoint, we obtain
\begin{eqnarray*}
\|h\|_{\ell_r}\mathbf{1}_{A_j} &=& \frac{\|\eta\|_{\ell_{r'}}^{p'-1}}{\|\eta\|_{\mathcal{K}(L_{p',q'},\ell_{r'})}^{q'-1}}\sum_{i\ge0}\frac{\mathbf{1}_{A_i}}{\left\|\|\eta\|_{\ell_{r'}}\mathbf{1}_{A_i}\right\|_{L_{p'}}^{p'-q'}}\mathbf{1}_{A_j}\\
&=& \frac{\|\eta\|_{\ell_{r'}}^{p'-1}}{\|\eta\|_{\mathcal{K}(L_{p',q'},\ell_{r'})}^{q'-1}}\frac{\mathbf{1}_{A_j}}{\left\|\|\eta\|_{\ell_{r'}}\mathbf{1}_{A_j}\right\|_{L_{p'}}^{p'-q'}}.
\end{eqnarray*}
We now take the $L_p(\mathbb{R})$-norm of both sides. 
\begin{eqnarray*}
\int_{\mathbb{R}}\|h\|^p_{\ell_r}\mathbf{1}_{A_j}\mathrm{d}\mathbb{P} &= &\int_{\mathbb{R}}\frac{\|\eta\|_{\ell_{r'}}^{(p'-1)p}}{\|\eta\|_{\mathcal{K}(L_{p',q'},\ell_{r'})}^{(q'-1)p}}\frac{\mathbf{1}_{A_j}}{\left\|\|\eta\|_{\ell_{r'}}\mathbf{1}_{A_j}\right\|_{L_{p'}}^{(p'-q')p}}\mathrm{d}\mathbb{P}\\
& = &\frac{1}{\|\eta\|_{\mathcal{K}(L_{p',q'},\ell_{r'})}^{(q'-1)p}}\int_{\mathbb{R}}\frac{\|\eta\|_{\ell_{r'}}^{(p'-1)p}\mathbf{1}_{A_j}}{\left\|\|\eta\|_{\ell_{r'}}\mathbf{1}_{A_j}\right\|_{L_{p'}}^{(p'-q')p}}\mathrm{d}\mathbb{P}\\
& = & \frac{1}{\|\eta\|_{\mathcal{K}(L_{p',q'},\ell_{r'})}^{(q'-1)p}}\frac{1}{\left\|\|\eta\|_{\ell_{r'}}\mathbf{1}_{A_j}\right\|_{L_{p'}}^{(p'-q')p}}\int_{\mathbb{R}}\|\eta\|_{\ell_{r'}}^{p'}\mathbf{1}_{A_j}\mathrm{d}\mathbb{P}\\
& = &  \frac{1}{\|\eta\|_{\mathcal{K}(L_{p',q'},\ell_{r'})}^{(q'-1)p}}\frac{\left\|\|\eta\|_{\ell_{r'}}\mathbf{1}_{A_j}\right\|_{L_{p'}}^{p'}}{\left\|\|\eta\|_{\ell_{r'}}\mathbf{1}_{A_j}\right\|_{L_{p'}}^{(p'-q')p}}\\
& = & \frac{1}{\|\eta\|_{\mathcal{K}(L_{p',q'},\ell_{r'})}^{(q'-1)p}}\left\|\|\eta\|_{\ell_{r'}}\mathbf{1}_{A_j}\right\|_{L_{p'}}^{(q'-1)p}.
\end{eqnarray*}
Therefore $$\left\|\|h\|_{\ell_r}\mathbf{1}_{A_j}\right\|_{L_p}=  \frac{1}{\|\eta\|_{\mathcal{K}(L_{p',q'},\ell_{r'})}^{q'-1}}\left\|\|\eta\|_{\ell_{r'}}\mathbf{1}_{A_j}\right\|_{L_{p'}}^{q'-1}.$$ Hence 
\begin{eqnarray*}
\sum_{j\ge0}\left\|\|h\|_{\ell_r}\mathbf{1}_{A_j}\right\|^q_{L_p} &=&  \sum_{j\ge0}\frac{1}{\|\eta\|_{\mathcal{K}(L_{p',q'},\ell_{r'})}^{(q'-1)q}}\left\|\|\eta\|_{\ell_{r'}}\mathbf{1}_{A_j}\right\|_{L_{p'}}^{(q'-1)q}\\
&=&\frac{1}{\|\eta\|_{\mathcal{K}(L_{p',q'},\ell_{r'})}^{q'}} \sum_{j\ge0}\left\|\|\eta\|_{\ell_{r'}}\mathbf{1}_{A_j}\right\|_{L_{p'}}^{q'}
 \end{eqnarray*}
 and then $$\sum_{j\ge0}\left\|\|h\|_{\ell_r}\mathbf{1}_{A_j}\right\|^q_{L_p}=\frac{1}{\|\eta\|_{\mathcal{K}(L_{p',q'},\ell_{r'})}^{q'}} \|\eta\|_{\mathcal{K}(L_{p',q'},\ell_{r'})}^{q'}=1.$$
Therefore $$\|h\|_{\mathcal{K}(L_{p,q},\ell_r)}=\left(\sum_{j\ge0}\left\|\|h\|_{\ell_r}\mathbf{1}_{A_j}\right\|^q_{L_p}\right)^{\frac{1}{q}}=1.$$ Thus $h=(h_k)_{k\ge0}\in\mathcal{K}(L_{p,q},\ell_r)$ since $h=(h_k)_{k\ge0}$ is measurable. 
\end{proof}
The following Proposition characterizes the dual of $\mathcal{K}(L_{p,q},\ell_r).$
\begin{proposition}\label{pro1}
For $1< p,q<\infty$ and $1\le r<\infty$, the dual space, $\mathcal{K}(L_{p,q},\ell_r)^*,$ of $\mathcal{K}(L_{p,q},\ell_r)$ is $\mathcal{K}(L_{p',q'},\ell_{r'})$ where $$\frac{1}{p} +\frac{1}{p'}=1,\quad\frac{1}{q} +\frac{1}{q'}=1,\quad\frac{1}{r} +\frac{1}{r'}=1.$$
\end{proposition}
\begin{proof}
Let $\eta=(\eta_k)_{k\ge0}\in\mathcal{K}(L_{p',q'},\ell_{r'})$ and $\epsilon=(\epsilon_k)_{k\ge0}\in\mathcal{K}(L_{p,q},\ell_r).$ Let $\langle\cdot,\cdot\rangle$ be the usual inner product, that is, $$\langle\eta,\epsilon\rangle=\sum_k\eta_k\epsilon_k.$$ 
and define the functional, $\Lambda,$ by $$\Lambda_{\eta}(\epsilon)=\mathbb{E}\langle\eta,\epsilon\rangle=\int_{\mathbb{R}}\sum_{k\ge0}\epsilon_k\eta_k\mathrm{d}\mathbb{P}=\sum_{j\ge0}\int_{A_j}\sum_{k\ge0}\epsilon_k\eta_k\mathrm{d}\mathbb{P}$$ for all $\eta=(\eta_k)_{k\ge0}\in\mathcal{K}(L_{p',q'},\ell_{r'})$ measurable and $\epsilon=(\epsilon_k)_{k\ge0}\in\mathcal{K}(L_{p,q},\ell_r).$ 
Then by H\"older inequality,
\begin{eqnarray}\label{gg1}
|\Lambda_{\eta}(\epsilon)|=\left|\sum_{j\ge0}\int_{A_j}\sum_{k\ge0}\epsilon_k\eta_k\mathrm{d}\mathbb{P}\right|\le \|\epsilon\|_{\mathcal{K}(L_{p,q},\ell_r)}\|\eta\|_{\mathcal{K}(L_{p',q'},\ell_{r'})}.
\end{eqnarray} 
and since $\Lambda_{\eta}(\cdot)$ is linear and bounded, it is a continuous linear functional on $\mathcal{K}(L_{p,q},\ell_r).$ From inequality (\ref{gg1}), we deduce that $\Lambda_\eta\in \left(\mathcal{K}(L_{p,q},\ell_r)\right)'$ and
\begin{eqnarray}\label{gg2}
\|\Lambda_{\eta}\|\le \|\eta\|_{\mathcal{K}(L_{p',q'},\ell_{r'})}.
\end{eqnarray}

For the converse, we can suppose that $q<p$. Let $\Lambda$ be a continuous linear functional on $\mathcal{K}(L_{p,q},\ell_r).$ Then as $\mathcal{K}(L_{p,q},\ell_r)$ embeds continuously into $\mathcal{K}(L_{p},\ell_r)$ (since $q<p$), we have by Hahn-Banach Theorem that $\Lambda$ can be extended to a continuous linear functional $\tilde{\Lambda}$ on $\mathcal{K}(L_{p},\ell_r)$ having the same operator norm as $\Lambda.$ It follows from (\cite[Lemma 2.9]{Ferenc}) that there exists some $\eta\in\mathcal{K}(L_{p'},\ell_{r'})$ such that $$\tilde{\Lambda}_{\eta}(\epsilon)=\mathbb{E}\langle\eta,\epsilon\rangle$$ for all $\epsilon\in\mathcal{K}(L_{p},\ell_r).$ In particular $$\Lambda_{\eta}(\epsilon)=\tilde{\Lambda}_{\eta}(\epsilon)=\mathbb{E}\langle\eta,\epsilon\rangle$$ for all $\epsilon\in\mathcal{K}(L_{p,q},\ell_r).$ Let us now show that $$\|\eta\|_{\mathcal{K}(L_{p',q'},\ell_{r'})}\lesssim \sup_{ \epsilon\in\mathcal{K}(L_{p,q},\ell_r),\, \|\epsilon\|_{\mathcal{K}(L_{p,q},\ell_r)}\le1}|\Lambda_{\eta}(\epsilon)|.$$ 
Set $h=(h_k)_{k\ge0}$ to be the sequence defined in Lemma \ref{lemma1}. Since $h=(h_k)_{k\ge0}\in\mathcal{K}(L_{p,q},\ell_r)$ with a unit norm, by linearity of the expectation operator, we have that for $1<r<\infty$,
\begin{eqnarray*}
\|\Lambda\| &\ge& |\Lambda_{\eta}(h)|=\frac{1}{\|\eta\|_{\mathcal{K}(L_{p',q'},\ell_{r'})}^{q'-1}}\mathbb{E}\left(\left(\sum_{k\ge0}|\eta_k|^{r'}\right)^{\frac{p'}{r'}}\sum_{j\ge0}\frac{\mathbf{1}_{A_j}}{\left\|\|\eta\|_{\ell_{r'}}\mathbf{1}_{A_j}\right\|_{L_{p'}}^{p'-q'}}\right) \\
&=&\frac{1}{\|\eta\|_{\mathcal{K}(L_{p',q'},\ell_{r'})}^{q'-1}}\mathbb{E}\left(\|\eta\|_{\ell_{r'}}^{p'}\sum_{j\ge0}\frac{\mathbf{1}_{A_j}}{\left\|\|\eta\|_{\ell_{r'}}\mathbf{1}_{A_j}\right\|_{L_{p'}}^{p'-q'}}\right) \\
&=&\frac{1}{\|\eta\|_{\mathcal{K}(L_{p',q'},\ell_{r'})}^{q'-1}}\sum_{j\ge0}\mathbb{E}\left(\frac{\|\eta\|_{\ell_{r'}}^{p'}\mathbf{1}_{A_j}}{\left\|\|\eta\|_{\ell_{r'}}\mathbf{1}_{A_j}\right\|_{L_{p'}}^{p'-q'}}\right) \\
&=&\frac{1}{\|\eta\|_{\mathcal{K}(L_{p',q'},\ell_{r'})}^{q'-1}}\sum_{j\ge0}\frac{\mathbb{E}(\|\eta\|_{\ell_{r'}}^{p'}\mathbf{1}_{A_j})}{\left\|\|\eta\|_{\ell_{r'}}\mathbf{1}_{A_j}\right\|_{L_{p'}}^{p'-q'}} \\
&=&\frac{1}{\|\eta\|_{\mathcal{K}(L_{p',q'},\ell_{r'})}^{q'-1}}\left[\sum_{j\in\mathbb{Z}}\frac{\left\|\|\eta\|_{\ell_{r'}}\mathbf{1}_{A_j}\right\|_{L_{p'}}^{p'}}{\left\|\|\eta\|_{\ell_{r'}}\mathbf{1}_{A_j}\right\|_{L_{p'}}^{p'-q'}}\right] \\
&=&\frac{1}{\|\eta\|_{\mathcal{K}(L_{p',q'},\ell_{r'})}^{q'-1}}\left[\sum_{j\in\mathbb{Z}}\left\|\|\eta\|_{\ell_{r'}}\mathbf{1}_{A_j}\right\|_{L_{p'}}^{q'}\right]. 
\end{eqnarray*}
Therefore
\begin{eqnarray*}
\|\Lambda\|&\ge&\frac{1}{\|\eta\|_{\mathcal{K}(L_{p',q'},\ell_{r'})}^{q'-1}}\|\eta\|_{\mathcal{K}(L_{p',q'},\ell_{r'})}^{q'} = \|\eta\|_{\mathcal{K}(L_{p',q'},\ell_{r'})}. 
\end{eqnarray*}
Thus
\begin{eqnarray}\label{gg3}
\| \Lambda \| \ge \|\eta\|_{\mathcal{K}(L_{p',q'},\ell_{r'})}
\end{eqnarray}
In the case of $r=1$, we note that there is and integer $k_0$ such that $$\frac 12\|\eta\|_{\ell_\infty}\le |\eta_{k_0}|.$$ Thus using the test function $h$ defined in Lemma \ref{lemma1} for $r=1$, and following the steps above, we obtain
\begin{eqnarray*}
\|\Lambda\| &\ge& |\Lambda_{\eta}(h)|=\frac{1}{\|\eta\|_{\mathcal{K}(L_{p',q'},\ell_{r'})}^{q'-1}}\mathbb{E}\left(\left(\sum_{k\ge0}\frac{|\eta_k|}{2^{k+1}}\right)\sum_{j\ge0}\frac{\|\eta\|_{\ell_{r'}}^{p'-1}\mathbf{1}_{A_j}}{\left\|\|\eta\|_{\ell_{r'}}\mathbf{1}_{A_j}\right\|_{L_{p'}}^{p'-q'}}\right) \\
&\ge&\frac{1}{2^{k_0+2}\|\eta\|_{\mathcal{K}(L_{p',q'},\ell_{r'})}^{q'-1}}\mathbb{E}\left(\|\eta\|_{\ell_{r'}}^{p'}\sum_{j\ge0}\frac{\mathbf{1}_{A_j}}{\left\|\|\eta\|_{\ell_{r'}}\mathbf{1}_{A_j}\right\|_{L_{p'}}^{p'-q'}}\right)\\ &=& \frac{1}{2^{k_0+2}}\|\eta\|_{\mathcal{K}(L_{p',q'},\ell_{r'})}.
\end{eqnarray*}
The proof is complete.
\end{proof}
We then see that $(\mathcal{K}(L_{p,q},\ell_{1}))^*=\mathcal{K}(L_{p',q'},\ell_{\infty})$ and thus it is now evident that the dual of the variation integrable space is the jump bounded space. More rigorously, we prove Theorem \ref{Garcia}.
\begin{proof}[Proof of Theorem \ref{Garcia}]
Let $g\in\mathcal{BD}_{p',q'}(\mathbb{R})$ and set  $$\kappa_g(f)=\sum_{k=1}^{\infty}\mathbb{E}[d_kfd_kg]\qquad\mbox{for}\,\,f\in\mathcal{G}_{p,q}(\mathbb{R}).$$ 
We obtain
\begin{eqnarray*}
|\kappa_g(f)|&=&\left|\sum_{k=1}^{\infty}\mathbb{E}[d_kfd_kg]\right|\\
&\le& \sum_{k=1}^{\infty}\mathbb{E}[|d_kf| |d_kg|] \le \mathbb{E}\sum_{k=1}^{\infty}[|d_kf| \sup_{k\in\mathbb{N}}|d_kg|]\\
&=&\int_{\mathbb{R}}\sum_{k=1}^{\infty}|d_kf| \sup_{k\in\mathbb{N}}|d_kg|\mathrm{d}\mathbb{P}\\
&=&\sum_{j\in\mathbb{Z}}\int_{A_j}\sum_{k=1}^{\infty}|d_kf| \sup_{k\in\mathbb{N}}|d_kg|\mathrm{d}\mathbb{P}\\
&\le&\sum_{j\in\mathbb{Z}}\left[\int_{A_j}\left(\sum_{k=1}^{\infty}|d_kf|\right)^p\mathrm{d}\mathbb{P}\right]^{\frac{1}{p}}\left[\int_{A_j} \sup_{k\in\mathbb{N}}|d_kg|^{p'}\mathrm{d}\mathbb{P}\right]^{\frac{1}{p'}}\\
&\le&\left\{\sum_{j\in\mathbb{Z}}\left[\int_{A_j}\left(\sum_{k=1}^{\infty}|d_kf|\right)^p\mathrm{d}\mathbb{P}\right]^{\frac{q}{p}}\right\}^{\frac{1}{q}}\left\{\sum_{j\in\mathbb{Z}}\left[\int_{A_j} \sup_{k\in\mathbb{N}}|d_kg|^{p'}\mathrm{d}\mathbb{P}\right]^{\frac{q'}{p'}}\right\}^{\frac{1}{q'}}.
\end{eqnarray*}
That is
\begin{eqnarray*}
|\kappa_g(f)|&\le&\left\|\sum_{n=0}^{\infty}|d_nf|\right\|_{p,q}\left\|\sup_{n\in\mathbb{N}}|d_ng|\right\|_{p',q'}=\|f\|_{\mathcal{G}_{p,q}(\mathbb{R})}\|g\|_{\mathcal{BD}_{p',q'}(\mathbb{R})}.
\end{eqnarray*}
Therefore $\kappa_g\in \left(\mathcal{G}_{p,q}(\mathbb{R})\right)'$ and $$\|\kappa_g\|\le \|g\|_{\mathcal{BD}_{p',q'}(\mathbb{R})}.$$

To prove the converse, we first assume that $\tau$ is an arbitrary element in the dual of $\mathcal{G}_{p,q}(\mathbb{R})$ then we show that there exists $g\in\mathcal{BD}_{p',q'}(\mathbb{R})$ such that $\tau=\kappa_g$ and $\|g\|_{\mathcal{BD}_{p',q'}}\le C\|\tau\|$ for some constant $C.$ 
\vskip .1cm
By setting $\epsilon_k=d_kf$ for $f\in\mathcal{M}$ we saw earlier that  $\mathcal{G}_{p,q}(\mathbb{R})\subseteq\mathcal{K}(L_{p,q},\ell_1).$ We also recall that the dual space of $\mathcal{K}(L_{p,q},\ell_1)$ is $\mathcal{K}(L_{p',q'},\ell_{\infty})$ and $\tau$ is a continuous linear functional on $\mathcal{G}_{p,q}(\mathbb{R})\subseteq\mathcal{K}(L_{p,q},\ell_1).$ By Hahn-Banach Theorem, $\tau$ can be extended to a continuous linear functional on $\mathcal{K}(L_{p,q},\ell_1)$ having the same operator norm as $\tau.$ Let $\Lambda$ be this extension of $\tau.$ Then we have by Proposition \ref{pro1} that there exists $\eta\in\mathcal{K}(L_{p',q'},\ell_{\infty})$ such that $$\|\Lambda\|=\|\tau\|\approx\|\eta\|_{\mathcal{K}(L_{p',q'},\ell_{\infty})}\quad\mbox{and}\quad\Lambda_{\eta}(\epsilon)=\sum_{k\ge0}\mathbb{E}(\epsilon_k\eta_k)$$ for $\epsilon\in\mathcal{K}(L_{p,q},\ell_1).$ Hence 
\begin{eqnarray}\label{tau1}
\tau(f_n) = \sum_{k=1}^n\mathbb{E}[(d_kf)\eta_k]=\sum_{k=1}^n\mathbb{E}[(d_kf)(\mathbb{E}_k\eta_k-\mathbb{E}_{k-1}\eta_k)]
\end{eqnarray}
is well defined (we agree for a moment to work with $f_n$ as we will show that $f_n\to f$ in $\mathcal{G}_{p,q}$ as $n\to\infty$).
We shall now set $$g_n:=\left\{\begin{array}{l}\sum_{k=1}^n[\mathbb{E}_k\eta_k-\mathbb{E}_{k-1}\eta_k]\quad\textrm{if}\quad n\neq 0\\g_0=0\end{array}\right.$$ and show that $g_n$ is a bounded martingale. Indeed we observe that as $k\to\infty,\,\,\mathbb{E}_k\eta_k-\mathbb{E}_{k-1}\eta_k\to 0.$ Thus $g_n$ is finite (i.e. $\sup_n|g_n|<\infty$) and also $$\mathbb{E}_{n-1}(g_n-g_{n-1})=\mathbb{E}_{n-1}(\mathbb{E}_n\eta_n-\mathbb{E}_{n-1}\eta_n)=\mathbb{E}_{n-1}\eta_n-\mathbb{E}_{n-1}\eta_n=0.$$This means that $g_n$ is a martingale. Consequently, since $$\sup_{k\in\mathbb{N}}|d_kg|=\sup_{k\in\mathbb{N}}|\mathbb{E}_k\eta_k-\mathbb{E}_{k-1}\eta_k|\le\sup_{k\in\mathbb{N}}(|\mathbb{E}_k|\eta_k|+\mathbb{E}_{k-1}|\eta_k|)\le2\sup_{n\in\mathbb{N}}\mathbb{E}_n\left(\sup_{k\in\mathbb{N}}|\eta_k|\right),$$ we can invoke inequality (\ref{doob1}) to obtain 
\begin{eqnarray*}
\|\sup_{k\in\mathbb{N}}|d_kg|\|_{p',q'}&\le&2\left\|\sup_{n\in\mathbb{N}}\mathbb{E}_n\left(\sup_{k\in\mathbb{N}}|\eta_k|\right)\right\|_{p',q'}\le\frac{2p'}{p'-1} \sup_{k\in\mathbb{N}}\left\|\mathbb{E}_n\left(\sup_{k\in\mathbb{N}}|\eta_k|\right)\right\|_{p',q'}\\
&\le&\frac{2p'}{p'-1}\left\|\sup_{k\in\mathbb{N}}|\eta_k|\right\|_{p',q'}=\frac{2p'}{p'-1}\|\eta\|_{\mathcal{K}(L_{p',q'},\ell_{\infty})}.
\end{eqnarray*}
Hence $g\in\mathcal{BD}_{p',q'}(\mathbb{R})$ and $\|g\|_{\mathcal{BD}_{p',q'}(\mathbb{R})}\le C\|\tau\|$ since $\|\tau\|=\|\eta\|_{\mathcal{K}(L_{p',q'},\ell_{\infty})}.$ 
\vskip .1cm
We now show that $f_n\to f$ in $\mathcal{G}_{p,q}$ as $n\to\infty.$
We observe that since $f=\sum_{k=0}^{\infty}f_k-f_{k-1},$ we have that $$f_n-f = f_n-\sum_{k=0}^{\infty}f_k-f_{k-1}= f_n-\sum_{k=0}^nf_k-f_{k-1}-\sum_{k=n+1}^{\infty}f_k-f_{k-1}=-\sum_{k=n+1}^{\infty}f_k-f_{k-1}.$$ Hence for $n\to\infty,\,\,f_n-f\to0.$ Also
\begin{eqnarray*}
\sum_{k\ge0}|d_k(f_n-f)|&=&\sum_{k\ge0}|d_kf_n-d_kf|\le2\sum_{k\ge0}\left||d_kf_n|+|d_kf|\right|\le4\sum_{k\ge0}|d_kf|<\infty
\end{eqnarray*}
since $f\in\mathcal{G}_{p,q}$ and thus by the Dominated Convergence Theorem $$\|f_n-f\|_{\mathcal{G}_{p,q}(\mathbb{R})}=\left\|\sum_{k\ge0}|d_k(f_n-f)|\right\|_{L_{p,q}(\mathbb{R})}\to0\quad\mbox{for}\,\,n\to\infty.$$ That is $f_n\to f$ in $\mathcal{G}_{p,q}(\mathbb{R}).$ Consequently, from equation (\ref{tau1}), as $n\to\infty,$ we have that, by setting $\eta_k=d_kg,$ $$\tau(f_n) = \sum_{k=1}^n\mathbb{E}[(d_kf)\eta_k]\to\tau(f) = \sum_{k=1}^{\infty}\mathbb{E}[(d_kf)\eta_k]=\kappa_g(f).$$ Hence $\|g\|_{\mathcal{BD}_{p',q'}(\mathbb{R})}\le C\|\tau\|=C\|\kappa\|$ and the poof is complete.
\end{proof}

\section{Further discussions and conclusion}\label{sec6}
In this work, we were able to to characterize the dual of $H^s_{p,q}$ for $1<p\le q<\infty.$ We were also able to extend the Doob's inequality and the Burkholder-Davis-Gundy inequality from the classical Lebesgue space to the amalgam space.  We have discussed  the various relations between the martingale Hardy-amalgam spaces and showed that upon the condition of regular basis, the spaces are all equivalent. We have also introduced the variation integrable space relative to our setting and characterized its dual space. 
\vskip .2cm
We observe that the Davis' decomposition is also valid in the martingale Hardy-amalgam spaces. Indeed, we have the following decomposition which can be proved as in the case $p=q$ (see \cite[Lemma 2.13]{Ferenc})
\begin{theorem}\label{thm:dav1}
Let $f=\{f_n\}_{n\in\mathbb{N}}\in H^S_{p,q}\,\,(1\le p<\infty,\,\,0<q<\infty).$ Then there exists $h=\{h_n\}_{n\in\mathbb{N}}\in\mathcal{G}_{p,q}$ and $g=\{g_n\}_{n\in\mathbb{N}}\in\mathcal{Q}_{p,q}$ such that $f=\{f_n=h_n+g_n\}_{n\in\mathbb{N}}$ for all $n\in\mathbb{N}$ and $$\|h\|_{\mathcal{G}_{p,q}}\le (2+2C)\|f\|_{H^S_{p,q}}$$ and $$\|g\|_{\mathcal{Q}_{p,q}}\le (7+2C)\|f\|_{H^S_{p,q}}.$$
\end{theorem}
Combining the above theorem with the second of the inequalities (v) in Theorem \ref{thm:mt1}, we derive the following.
\begin{corollary}\label{cor:dav2}
Let $f=\{f_n\}_{n\in\mathbb{N}}\in H^S_{p,q}\,\,(1\le p<\infty,\,\,0<q<\infty).$ Then there exists $h=\{h_n\}_{n\in\mathbb{N}}\in\mathcal{G}_{p,q}$ and $g=\{g_n\}_{n\in\mathbb{N}}\in H^s_{p,q}$ such that $f=\{f_n=h_n+g_n\}_{n\in\mathbb{N}}$ for all $n\in\mathbb{N}$ and $$\|h\|_{\mathcal{G}_{p,q}}\le (2+2C)\|f\|_{H^S_{p,q}}$$ and $$\|g\|_{H^s_{p,q}}\le (7+2C)\|f\|_{H^S_{p,q}}.$$
\end{corollary}
The following Davis' decomposition of $H^*_{p,q}$ follows also as in \cite[Lemma 2.14]{Ferenc}.
\begin{theorem}\label{thm:dav2}
Let $f=\{f_n\}_{n\in\mathbb{N}}\in H^*_{p,q}\,\,(1\le p<\infty,\,\,0<q<\infty).$ Then there exists $h=\{h_n\}_{n\in\mathbb{N}}\in\mathcal{G}_{p,q}$ and $g=\{g_n\}_{n\in\mathbb{N}}\in\mathcal{P}_{p,q}$ such that $f=\{f_n=h_n+g_n\}_{n\in\mathbb{N}}$ for all $n\in\mathbb{N}$ and $$\|h\|_{\mathcal{G}_{p,q}}\le (4+4C)\|f\|_{H^*_{p,q}}$$ and $$\|g\|_{\mathcal{P}_{p,q}}\le (13+4C)\|f\|_{H^*_{p,q}}.$$
\end{theorem}
Combining the above theorem with inequality (v) in Theorem \ref{thm:mt1}, we derive the following.
\begin{corollary}\label{cor:dav2}
Let $f=\{f_n\}_{n\in\mathbb{N}}\in H^*_{p,q}\,\,(1\le p<\infty,\,\,0<q<\infty).$ Then there exists $h=\{h_n\}_{n\in\mathbb{N}}\in\mathcal{G}_{p,q}$ and $g=\{g_n\}_{n\in\mathbb{N}}\in H^s_{p,q}$ such that $f=\{f_n=h_n+g_n\}_{n\in\mathbb{N}}$ for all $n\in\mathbb{N}$ and $$\|h\|_{\mathcal{G}_{p,q}}\le (4+4C)\|f\|_{H^*_{p,q}}$$ and $$\|g\|_{H^s_{p,q}}\le (13+4C)\|f\|_{H^*_{p,q}}.$$
\end{corollary}
We close this work with the proof of the following duality result that extend the one of \cite[Theorem 2.34]{Ferenc}.
\begin{theorem}
The dual space of $H^*_{p,q}\,\,(1< q\le p\le2)$ can be given with the norm $$\|\phi\|:=\|\phi\|_{H^s_{p',q'}}+\|\phi\|_{\mathcal{BD}_{p'q'}}$$ where $\frac{1}{p}+\frac{1}{p'}=\frac{1}{q}+\frac{1}{q'}=1.$
\end{theorem}
\begin{proof}
Let $\phi\in H^s_{p',q'}\cap\mathcal{BD}_{p',q'}.$ Then $\phi\in L_2$ since $L_2$ is dense in $H^s_{p',q'}$ (from its atomic decomposition). Define the functional $\kappa_{\phi}$ by 
\begin{eqnarray}\label{H*linear1}
\kappa_{\phi}(f) = \mathbb{E}(f\phi),\quad (f\in L_2).
\end{eqnarray}
We will show that (\ref{H*linear1}) is a bounded linear functional on $H^*_{p,q}\,\,(1< p\le q\le2).$ 
\vskip .1cm
Linearity follows since the expectation operator is linear. Also as $L_2$ is dense in $H^*_{p,q}$, we have that $\kappa_\phi$ is well defined. 
As $f_n\to f$ in $L_2$ norm (as $n\to\infty$), we have that $$\kappa_{\phi}(f) := \mathbb{E}(f\phi)=\lim_{n\to\infty}\mathbb{E}(f_n\phi).$$ Now for $f\in H^*_{p,q},$ we know from Davis' decomposition (Corollary \ref{cor:dav2}) that $f_n=h_n+g_n$ where $h=\{h_n\}_{n\in\mathbb{N}}$ and $g=\{g_n\}_{n\in\mathbb{N}}$ are martingales 
such that \begin{eqnarray}\label{H*decomp}\|h\|_{\mathcal{G}_{p,q}}\lesssim \|f\|_{H^*_{p,q}}\end{eqnarray} and \begin{eqnarray}\label{H*decomp1}\|g\|_{H^s_{p,q}}\lesssim \|f\|_{H^*_{p,q}}.\end{eqnarray} Hence we have by linearity of $\mathbb{E}$ that 
\begin{eqnarray}\label{explinearity}
|\mathbb{E}(f_n\phi)| = |\mathbb{E}(h_n\phi+g_n\phi)|\le |\mathbb{E}(g_n\phi)| + |\mathbb{E}(h_n\phi)|
\end{eqnarray}
From (\ref{H*decomp1}) we have that $g\in H^s_{p,q}$ since $f\in H^*_{p,q}.$ Hence since $\phi\in H^s_{p',q'},$ it follows from Theorem \ref{thm:dualitygrand} for $(q\le p)$ that
\begin{eqnarray}\label{Hslinear}
|\mathbb{E}(g_n\phi)| \le \|g_n\|_{H^s_{p,q}}\|\phi\|_{H^s_{p',q'}}.
\end{eqnarray}
Similarly, since $h\in\mathcal{G}_{p,q}$ and $\phi\in\mathcal{BD}_{p',q'},$ we have by Theorem 4.13 that
\begin{eqnarray}\label{Glinear}
|\mathbb{E}(h_n\phi)| \le \|h_n\|_{\mathcal{G}_{p,q}}\|\phi\|_{\mathcal{BD}_{p',q'}}.
\end{eqnarray}
Therefore (\ref{explinearity}) becomes 
\begin{eqnarray*}\label{explinearitynormn}
|\mathbb{E}(f_n\phi)| \le  \|g_n\|_{H^s_{p,q}}\|\phi\|_{H^s_{p',q'}} + \|h_n\|_{\mathcal{G}_{p,q}}\|\phi\|_{\mathcal{BD}_{p',q'}}
\end{eqnarray*}
and thus 
\begin{eqnarray*}
|\mathbb{E}(f\phi)| \le  \|g\|_{H^s_{p,q}}\|\phi\|_{H^s_{p',q'}} + \|h\|_{\mathcal{G}_{p,q}}\|\phi\|_{\mathcal{BD}_{p',q'}}.
\end{eqnarray*}
It then follows from (\ref{H*decomp}) and (\ref{H*decomp1}) that 
\begin{eqnarray*}
|\mathbb{E}(f\phi)| \le  \|f\|_{H^*_{p,q}}\|\phi\|_{H^s_{p',q'}} + \|f\|_{H^*_{p,q}}\|\phi\|_{\mathcal{BD}_{p',q'}}.
\end{eqnarray*}
In other words, 
\begin{eqnarray}\label{explinearitynorm}
|\mathbb{E}(f\phi)| \le  \|f\|_{H^*_{p,q}}(\|\phi\|_{H^s_{p',q'}} + \|\phi\|_{\mathcal{BD}_{p',q'}}).
\end{eqnarray}
Thus the functional $\kappa_{\phi}$ is continuous linear functional on $H^*_{p,q}.$

Conversely, assume that $\kappa_{\phi}$ is an arbitrary continuous linear on $H^*_{p,q}.$ Then as $H^*_{p,q}$ embeds continuously in $H^*_p\,\,(q<p),$ we have by Hahn-Banach theorem that $\kappa_{\phi}$ can be extended to a continuous linear functional $\tilde{\kappa_{\phi}}$ on $H^*_p$ having the same operator norm as $\kappa_{\phi}.$ It follows from \cite[Theorem 2.34]{Ferenc} that there exists some $\phi\in L_2$ such that $$\tilde{\kappa_{\phi}}(f)=\mathbb{E}(f\phi),\,\,(f\in L_2).$$ In particular $$\kappa_{\phi}(f)=\tilde{\kappa_{\phi}}(f)=\mathbb{E}(f\phi),\,\,(f\in H^*_{p,q}).$$ We observe that since 
\begin{eqnarray*}
f^*:= \sup_{n\in\mathbb{N}}|f_n| &=& \sup_{n\in\mathbb{N}}\left|\sum_{k=1}^nd_kf\right| \\
&\le& \sup_{n\in\mathbb{N}}\sum_{k=1}^n|d_kf| \le \sup_{n\in\mathbb{N}}\sum_{k=1}^{\infty}|d_kf|\\ &=& \sum_{k=1}^{\infty}|d_kf|,
\end{eqnarray*}
we have that $$\|f^*\|_{p,q}\le\left\|\sum_{k=1}^{\infty}|d_kf|\right\|_{p,q}\quad (\mbox{in other words},\,\, \|f\|_{H^*_{p,q}}\le\|f\|_{\mathcal{G}_{p,q}}).$$ Thus $\mathcal{G}_{p,q}\subseteq H^*_{p,q}.$ Therefore $\kappa_{\phi}$ is also a continuous linear functional on $\mathcal{G}_{p,q}.$ It follows from Theorem \ref{Garcia} that
\begin{eqnarray}\label{H*converse}
\|\phi\|_{\mathcal{BD}_{p',q'}}\le C\|\kappa_{\phi}\|.
\end{eqnarray}
We also recall that from Theorem \ref{thm:mt1}, we have $\|f\|_{H^*_{p,q}}\le C\|f\|_{H^s_{p,q}}\,\,(1\le p,q\le2)$  Therefore $\kappa_{\phi}$ is also a continuous linear functional on $H^s_{p,q}.$ Hence by Theorem \ref{thm:dualitygrand}, for $(q\le p)$,  
 \begin{eqnarray}\label{H*converse1}
\|\phi\|_{H^s_{p',q'}}\le C\|\kappa_{\phi}\|.
\end{eqnarray}
From (\ref{H*converse}) and (\ref{H*converse1}), we obtain that $$\|\phi\|_{\mathcal{BD}_{p',q'}}+\|\phi\|_{H^s_{p',q'}}\le C\|\kappa_{\phi}\|$$ and the proof is complete.
\end{proof}
Finally, we note that even in the dyadic case, the characterization of the dual of $H^s_{p,q}$ for $(p,q)\notin\{(p,q)\in [1,\infty)^2: 1<q\le p\le 2\quad\textrm{or}\quad 2\le p\le q<\infty\}$ is still open.
\section{Declarations}
The authors declare that they have no conflict of interest regarding this work.

\end{document}